\documentclass[a4paper,sort&compress]{amsart}
\usepackage{amssymb}
\usepackage{amsthm}
\usepackage{breqn}
\usepackage{enumerate}
\usepackage{mathrsfs}
\usepackage{pgfplots}
\usepackage{physics}

\usepackage{algorithm}
\usepackage{booktabs}

\usepackage[sort]{cite}

\DeclareRobustCommand\genfrac[6]{%
  {%
    \if 0#4\relax\displaystyle\else
    \if 1#4\relax\textstyle\else
    \if 2#4\relax\scriptstyle\else
    \if 3#4\relax\scriptscriptstyle\else
      #4%
    \fi\fi\fi\fi
    \fracstyle
    {\begingroup #5\endgroup
      \csname @@\ifx\maxdimen#3\maxdimen over\else above\fi
        \if @#1@\else withdelims\fi\endcsname #1 #2 #3\relax
     #6}%
  }%
}
\renewcommand{\binom}{\genfrac(){0pt}{}}

\usepackage{xcolor}
\usepackage{tikz}
\usetikzlibrary{backgrounds,patterns,matrix,calc,arrows}

\usepackage{hyperref}
\usepackage[capitalize]{cleveref}
\Crefname{equation}{}{}

\title[Nondefective secant varieties of the Chow variety]{All secant varieties of the Chow variety are nondefective for cubics and quaternary forms}
\date{}

\author{Douglas A. Torrance}
\email{dtorrance@piedmont.edu}
\address{Piedmont College, Georgia, United States of America.}

\author{Nick Vannieuwenhoven}
\email{nick.vannieuwenhoven@kuleuven.be}
\address{KU Leuven, Department of Computer Science, Leuven, Belgium.}
\thanks{NV was supported by a Postdoctoral Fellowship of the Research Foundation---Flanders (FWO) with project 12E8119N}


\newcommand{\GG}{\mathbb G}
\newcommand{\kk}{\Bbbk}
\newcommand{\NN}{\mathbb N}
\newcommand{\PP}{\mathbb P}

\newcommand{\ZZ}{\mathbb Z}
\newcommand{\CC}{\mathbb C}

\DeclareMathAlphabet{\pzc}{OT1}{pzc}{m}{it}
\newcommand{\statement}[1]{\pzc{{#1}}}
\newcommand{\mx}[1]{\mathrm{\mathbf{#1}}}

\DeclareMathOperator{\CV}{CV}
\DeclareMathOperator{\LC}{LC}
\DeclareMathOperator{\expdim}{expdim}
\DeclareMathOperator{\Seg}{Seg}
\DeclareMathOperator{\Split}{Split}

\newtheorem{theorem}{Theorem}[section]
\newtheorem{proposition}[theorem]{Proposition}
\newtheorem{lemma}[theorem]{Lemma}
\newtheorem{corollary}[theorem]{Corollary}
\newtheorem{conjecture}[theorem]{Conjecture}
\theoremstyle{definition}
\newtheorem{definition}[theorem]{Definition}
\newtheorem{example}[theorem]{Example}

\numberwithin{equation}{section}

\begin{document}

\begin{abstract}
  The Chow rank of a form is the length of its smallest decomposition
  into a sum of products of linear forms.
  For a generic form, this corresponds to finding the smallest secant
  variety of the Chow variety which fills the ambient space.
  We determine the Chow rank of generic cubics and quaternary forms by proving nondefectivity of all involved secant varieties.
  The main new ingredient in our proof is the generalization of a technique by [Brambilla and Ottaviani, On the Alexander--Hirschowitz theorem, J. Pure Appl. Algebra, 2008] that consists of employing Terracini's lemma and Newton's backward difference formula to compute the dimensions of secant varieties of arbitrary projective varieties. Via this inductive construction, the proof of nondefectivity ultimately reduces to proving a number of base cases. These are settled via a computer-assisted proof because of the large dimensions of the spaces involved. The largest base case required in our proof consisted of computing the dimension of a vector space constructed from the $400$th secant variety of a degree-$82$ Chow variety embedded in $\mathbb{P}^{98769}$.
\end{abstract}

\maketitle

\section{Introduction}

A famous question in number theory is \textit{Waring's problem}:
Given any $d\in\NN$, find the smallest $s$ such that for
all $n\in\NN$, there exist $n_1,\ldots,n_s\in\NN$ such that
\begin{equation*}
  n = n_1^d + \cdots + n_s^d.
\end{equation*}
The fact that such an $s$ exists for each $d$ was stated in \cite{waring} by Waring
himself without proof in 1770 and was finally proven by
Hilbert 139 years later \cite{hilbert}.
According to \cite{BO}, still about two decades before Hilbert's definite answer, a substantial generalization of Waring's problem was studied in 1891 by Campbell \cite{Campbell1891}. He studied the question: given any $n,d\in\NN$, what is the smallest $s\in\NN$ so that a homogeneous polynomial of degree $d$ in $n+1$ variables $f$ can be expressed as
\begin{equation}\label{eqn_waring_decomposition}
  f = \ell_1^d + \cdots + \ell_s^d,
\end{equation}
for some linear forms $\ell_1,\ldots,\ell_s$? This $s$ is called the \textit{Waring rank} of $f$. The Waring rank of \textit{generic} polynomials $f$ was ultimately determined a century later in 1995 by Alexander and Hirschowitz \cite{AH}. For brevity, we use the standard terminology ``generic'' to mean ``outside of a closed set in the Zariski topology'' in this paper.

In recent years, tensors and their decompositions have witnessed a tremendous increase in popularity in applied mathematics, chemometrics, psychometrics, signal processing, and machine learning \cite{Kolda2009}. The expression \cref{eqn_waring_decomposition} is analoguous to a standard tensor decomposition. Recall that homogeneous polynomials of degree $d$ in $n+1$ variables form a vector space that corresponds to the subspace of symmetric tensors $S^d \CC^{n+1} \subset \CC^{n+1} \otimes \cdots \otimes \CC^{n+1}$. After choosing coordinates on $\CC^{n+1}$ and taking the $d$-fold tensor product, a symmetric tensor $f \in S^d \CC^{n+1}$ can be represented by a $d$-array $F$ with a full symmetry, by which we mean that $F_{i_1, \ldots, i_d} = F_{\sigma(i_1), \ldots, \sigma(i_d)}$ for every permutation $\sigma$ on $d$ elements. 
In this terminology, the Waring rank of $f$ in \cref{eqn_waring_decomposition} corresponds to the smallest $s$ such that 
\[
 F = \sum_{i=1}^s {L}_i \otimes \cdots \otimes {L}_i,
\]
where $L_i \in \CC^{n+1}$ and $\otimes$ is the tensor product. See \cite{CGLM2008} for details on this viewpoint.

This paper studies a generalization of foregoing \emph{Waring decomposition} from \cref{eqn_waring_decomposition}.
In this case, we have a partition of $d$, i.e., $\vb d = (d_1,\ldots, d_k)$ with $||\vb d||_1=d$, where $||\vb d||_1=d_1+\cdots+d_k$ is the usual $L^1$-norm, and we seek to express a degree-$d$ form $f$ in $n+1$ variables as a $\vb d$-Chow--Waring decomposition \cite{AV,CCGO} with a minimal number of terms $s$. That is,
\begin{equation}\label{chow-waring decomposition}
  f = \ell_{1,1}^{d_1}\cdots\ell_{1,k}^{d_k} + \cdots +
  \ell_{s,1}^{d_1}\cdots\ell_{s,k}^{d_k}.
\end{equation}
where $\ell_{1,1},\ldots,\ell_{1,k},\ldots,\ell_{s,1},\ldots,\ell_{s,k}$ are linear forms and $s$ is minimal.
If $\vb d = (d)$, then $s$ is exactly the Waring rank of $f$. On the
other extreme, if $\vb d = (1,\ldots,1)$, then this value of $s$ has become known
as the \textit{Chow rank} of $f$. For arbitrary $\vb d$, the above minimal $s$ is the $\vb d$th \textit{Chow--Waring rank} of $f$. The novel contribution of this work concerns the question: given a generic $f \in S^d \CC^{n+1}$, what is its Chow rank?

These types of questions are naturally studied using algebraic geometry; see \cite{BCCGO} for an excellent overview.
In this light, the central problem of this paper can be interpreted as finding the smallest $s$ such that the $s$th \textit{secant variety} of the \textit{Chow variety} (or \textit{split variety}) $\mathcal{C}_{d,n}$ of completely decomposable forms fills the ambient $\PP^{\binom{n+d}{d}-1}$. In general, the $\vb d$th Chow--Waring rank of a generic form is the smallest $s$ for which $\sigma_s(\CV_{\vb d}(\PP^n))$, the $s$th secant variety of the \textit{Chow--Veronese variety} $\CV_{\vb d}(\PP^n)$ of polynomials of the form $\ell_1^{d_1}\cdots\ell_k^{d_k}$, fills the ambient space. Note that $\CV_{(d)}(\PP^n)=v_d(\PP^n)$, the $d$th Veronese embedding of $\PP^n$, and $\CV_{(1,\ldots,1)}(\PP^n)=\mathcal{C}_{d,n}$.
 

Based on a na\"ive parameter count, it is expected that secant
varieties of the $\vb d$th Chow--Veronese variety have dimension
\begin{equation*}
  \expdim\sigma_s(\CV_{\vb d}(\PP^n))=\min\left\{s(kn+1),
    \binom{n+d}{d}\right\}-1,
\end{equation*}
and, consequently, the corresponding $\vb d$th Chow--Waring rank of a generic form $f$
would be $r_{exp} = \left\lceil(kn+1)^{-1}\binom{n+d}{d}\right\rceil$, which is the smallest
$s$ so that $s(kn+1)\geq\binom{n+d}{d}$.
However, this is not always the case.
A number of \textit{defective} examples exist in which $\dim\sigma_s(\CV_{\vb
d}(\PP^n))<\expdim\sigma_s(\CV_{\vb d}(\PP^n))$. All the known defective cases are listed in \cref{defective cases}. Aside from the case of quadrics, i.e., $||\vb d||_1=2$, only a finite number of secant varieties appears to be defective. As a defective $s_0$th secant variety with $s_0 < r_{exp}$ necessarily implies defectivity for all $s_0 \le s < r_{exp}$, it should be easy to find defective cases. Therefore, it is reasonable to make the following conjecture.

\begin{table}[tb]
\centering
\caption{Known defective cases}
\label{defective cases}
\begin{tabular}{ccccc}
  \toprule
  $\vb d$ & $n$ & $s$ & $\dim\sigma_s(\CV_{\vb d}(\PP^n))$ & Reference \\
  \midrule
  (2) & $\geq 2$ & $2,\ldots,n$ &
    $\binom{n + 2}{2} - \binom{n - s + 2}{2} - 1$ & \cite{BO} \\
  (1,1) & $\geq 4$ & $2,\ldots,\left\lfloor\frac{n}{2}\right\rfloor$ &
    $\binom{n + 2}{2} - \binom{n - 2s + 2}{2} - 1$ & \cite{AB,CGG,torrance}\\
  (3) & 4 & 7 & 33 & \cite{BO}\\
  (2, 1) & 2 & 2 & 8 & \cite{CCGO,CGG} \\
  (2, 1) & 3 & 3 & 18 & \cite{CGG} \\
  (2, 1) & 4 & 4 & 33 & \cite{CGG} \\
  (4) & 2 & 5 & 13 & \cite{BO}\\
  (4) & 3 & 9 & 33 & \cite{BO}\\
  (4) & 4 & 14 & 68 & \cite{BO}\\
  \bottomrule
\end{tabular}
\end{table}

\begin{conjecture}\label{conjecture}
  With the exception of the known defective cases in \cref{defective cases}, we have $\dim\sigma_s(\CV_{\vb
    d}(\PP^n))=\expdim\sigma_s(\CV_{\vb d}(\PP^n))$ for all $\vb d$, $n$,
  and $s$.
\end{conjecture}

Substantial progress has been made towards proving this conjecture.
Most famously, the Veronese case, $\vb d = (d)$, was completed by
Alexander and Hirschowitz in \cite{AH}; see \cite{BO} for an excellent overview of this case.
More recently, in \cite{AV}, Abo and the second author completed the
$\vb d = (d - 1, 1)$ case, for which the Chow--Veronese variety is the
tangential variety to a Veronese variety, building on foundational work by
Bernardi, Catalisano, Gimigliano, and Id\'a \cite{BCGI}.
Catalisano, Chiantini, Geramita, and Oneto also proved several results
for general $\vb d$ in \cite{CCGO}.

The remaining progress has been made in the Chow case with
$\vb d =(1,\ldots,1)$.
Arrondo and Bernardi were among the first to look at this case in \cite{AB}.
Shin then found a connection with Hilbert functions of unions of linear
star-configurations and was able to use this to complete the $n=2$
case for $d\leq 5$ \cite{shin1}.
Abo \cite{abo} then adapted a technique of Brambilla and Ottaviani \cite{BO} to
complete the $n=2$ case and make significant progress towards the
$n=3$ and $d=3$ cases.
The first author improved on this slightly in \cite{thesis}, and also
found in \cite{torrance} that \cref{conjecture} is true for the Chow case
provided that $s\leq 35$.
In \cite{CGGHMNS}, Catalisano et al.~examined secant varieties of
varieties of reducible hypersurfaces, i.e., products of forms of
arbitrary degree as opposed to only linear forms.
The overlap between their problem and the Chow--Waring problem is the
Chow variety case, and their results show that
$\sigma_s(\mathcal{C}_{d,n})$ has the expected dimension provided that
$s\geq\binom{n+d-1}{n}$.

In conclusion, \cref{nondefective cases} summarizes all the cases known to us for which \cref{conjecture} holds (excluding the exceptions in \cref{defective cases}).

\begin{table}[tb]
  \centering
  \caption{Known cases where Conjecture \ref{conjecture} holds.
  More details on the definitions of $s_1$, $s_2$, $s_1'$, and $s_2'$ can be found in \cref{old n=3 case} and \cref{old d=3 case}.}
  \label{nondefective cases}
  \begin{tabular}{cccc}
    \toprule
    $\vb d$ & $n$ & $s$ & Reference \\
    \midrule
    any & 1 & any & \cite{CCGO} \\
    any & any & $\leq\max\left\{2,\left\lfloor\frac{d+1}{k+1}\right\rfloor\left\lfloor\frac{n}{2}\right\rfloor,2\left\lfloor\frac{n}{3}\right\rfloor\right\}$
                        & \cite{CCGO} \\
    $(d)$ & any & any & \cite{AH,BO} \\
    $(d-1,1)$ & any & any & \cite{AV,BCGI} \\
    $(1,1,1)$ & any & $\leq s_1(n)$ or $\geq s_2(n)$ &
                                                                                  \cite{abo} \\
    $(1,\ldots,1)$ & 2 & any & \cite{abo,shin1} \\
    $(1,\ldots,1)$ & 3 & $\leq s_1'(d)$ or $\geq\min\{s_2(d),s_2'(d)\}$
                         & \cite{abo,thesis} \\
    $(1,\ldots,1)$ & any & $\leq\max\{35, s_1'(d)\}$ or
                           $\geq \binom{n+d-1}{n}$ &
                                                     \cite{CGGHMNS,torrance} \\
    \bottomrule
  \end{tabular}
\end{table}


In addition to the problem of computing the Chow rank of a generic
form, secant varieties of Chow varieties have applications to
complexity theory \cite{landsberg} and have connections to secant
varieties of Grassmannians \cite{AB}, unions of linear
star-configurations \cite{shin2,shin1}, and complete intersections on
hypersurfaces \cite{CCG1,CCG2}.

\subsection{Main contribution}
The novel contribution of this paper concerns resolving the Chow case of \cref{conjecture} when either $d=3$ or $n=3$.
In particular, we complete Abo's partial results from \cite{abo} and prove that these Chow varieties are never defective. These new results were established using a novel general technique that we call \textit{Brambilla--Ottaviani lattices}; they are introduced in \cref{BO lattices} below. The main result is this:

\begin{theorem}\label{main_result}
  All secant varieties to Chow varieties of cubics and quaternary
  forms have the expected dimension. That is, \cref{conjecture}
  is true for all $n$ if $\vb d=(1,1,1)$, and it is true for all $\vb d =
  (1,\ldots,1)$ if $n=3$.
\end{theorem}

\begin{corollary}
  The Chow rank of a generic $(n+1)$-ary cubic is
  $\left\lceil\binom{n+3}{3}/(3n+1)\right\rceil$ and the Chow rank of a
  generic quarternary $d$-ic is
  $\left\lceil\binom{d+3}{d}/(3d+1)\right\rceil$.
\end{corollary}

In \cite[Theorem 1.3]{torrance}, it is established that if
$s(3d+1)\leq\binom{d+3}{d}$ and $\sigma_s(\mathcal{C}_{d,3})$ is
nondefective, then $\sigma_s(\mathcal{C}_{d,n})$ is nondefective for all
$n\geq 3$.
Therefore we have an improved upper bound on $s$ for all $n$. 

\begin{corollary}
  If $s\leq\frac{1}{3d+1}\binom{d+3}{3}$, then all $s$th secant
varieties to Chow varieties of $d$-ics have the expected dimension.
\end{corollary}

Finally, since the cases $\vb d = (2), (1,1), (3)$, and $(2,1)$ were already established in the literature, our result concludes the classification of defective Chow--Veronese varieties for $\|\vb d\|_1 \le 3$, i.e., quadrics and cubics.

\begin{corollary}
Conjecture \ref{conjecture} is true for $\|\vb d\|_1 \le 3$. In particular, secant varieties to Chow--Veronese varieties of cubics have the expected dimension except for the known defective cases.
\end{corollary}

\subsection{Outline} 
The format of the paper is as follows.
In \cref{secant varieties}, we introduce some notation and recall Terracini's famous lemma, which is vital for computing dimensions of secant varieties.
We continue in \cref{finite calculus} by summarizing some useful results from finite calculus whose connection to dimensions of intersections of generic linear subspaces are foundational to the Brambilla--Ottaviani lattices we introduce in \cref{BO lattices}. 
Such lattices generalize a technique from \cite{BO} that has subsequently been adapted by Abo and his students for a variety of related problems.
In \cref{chow BO lattices}, we focus on the Chow variety and survey previous results which rely on Brambilla--Ottaviani lattices.
Finally, in \cref{verification}, we describe the induction to prove \cref{main_result} as well as the process used to verify the base cases of this induction, completing the proof.

\section{Secant varieties and Terracini's lemma}
\label{secant varieties}

Suppose $U\subset V$ are vector spaces of dimensions $M+1$ and $N+1$, respectively. Then $\PP U$ is a \textit{linear subspace} or $M$-\textit{plane} in the projective space $\PP V=\PP^N$. 
Conversely, if $P\subset\PP^N$ is a linear subspace and $P=\PP U$ for some vector space $U$, then $U=\widehat P$, the \textit{affine cone} of $X$.
If $X_1,\ldots,X_s\subset\PP^N$ are projective varieties, then their \textit{linear span} is the smallest linear subspace of $\PP^N$ containing their union, denoted $\langle X_1,\ldots,X_s\rangle$.

\begin{definition}
  Suppose $X\subset\PP^N$ is a projective variety.  Its $s$th \textit{secant variety} is
  \begin{equation*}
    \sigma_s(X)=\overline{\bigcup_{p_1,\ldots,p_s\in X}\langle p_1,\ldots,p_s\rangle},
  \end{equation*}
  i.e., the Zariski closure of the union of all $(s-1)$-planes through $s$ points on $X$.
\end{definition}

By a straightforward dimension count, we see that
\begin{equation*}
  \dim\sigma_s(X)\leq\min\{s(\dim X + 1) - 1,N\}.
\end{equation*}

\begin{definition}\label{def_expdim}
  The \textit{expected dimension} of $\sigma_s(X)$, denoted $\expdim\sigma_s(X)$, is the right hand side of the above inequality.  If $\dim\sigma_s(X)=\expdim\sigma_s(X)$, then $\sigma_s(X)$ is \textit{nondefective}.  Otherwise, it is \textit{defective}.
\end{definition}

An extremely useful classical result in determining whether a given secant variety is defective is \textit{Terracini's lemma}, which reduces the problem to linear algebra.

\begin{lemma}[Terracini \cite{Terracini1911}]\label{terracini}
  Suppose $X$ is an irreducible projective variety. Let $p_1,\ldots,p_s\in X$ be generic points and suppose $q$ is a generic point in the $(s-1)$-plane spanned by $p_1,\ldots,p_s$.  Then
  \begin{equation*}
    T_q \widehat{\sigma_s(X)} = \sum_{j=1}^s T_{p_j} \widehat{X}.
  \end{equation*}
\end{lemma}

\section{Finite calculus}\label{finite calculus}
At a high level, Brambilla and Ottaviani's approach \cite{BO} for proving nondefectivity of secant varieties of third-order Veronese varieties consists of a three-step induction on the number of variables by partitioning $k_n$ generic points on $\mathcal{V}_{3,n} = v_3(\PP^n)$ and specializing them to the intersection of $v_3(\PP^n)$ with three special linear subspaces of codimension $\binom{n+3}{3}-\binom{n}{3}$. Terracini's lemma is then invoked at the specialized points to bound the dimension of $\sigma_s(\mathcal{V}_{3,n})$ from below. In \cite{BO} the main new idea was that this particular three-step induction ``has the advantage to avoid the arithmetic problems [that arise when specializing points].'' Indeed, with this setup, for $n \not\equiv 2 \mod 3$, we have that $k_n = \frac{1 + \dim \PP^n}{1 + \dim \mathcal{V}_{3,n}}$ is integer so that the $k_n$-secant variety of $\mathcal{V}_{3,n}$ is expected to precisely fill up the ambient space $\PP^{\binom{n+3}{3}}$ while $k_n-1$ points will not. This greatly simplifies the induction strategy because only three uniform specialization strategies are required for respectively $n = 3p + 0$, $n=3p+1$, and $n=3p+2$. ``This simple arithmetic remark'' of Brambilla and Ottaviani is a consequence of a more general connection between finite differences and intersections of linear subspaces, which we discuss next.

\begin{definition}\label{def_bd}
  Suppose $f:\ZZ\rightarrow\ZZ$ and fix a constant \textit{step size}
  $\ell$.  The \textit{backward difference operator} $\nabla$ is defined
  by
  \begin{align*}
    \nabla^0f(t) &= f(t)\\
    \nabla^1f(t) &= f(t) - f(t-\ell)\\
                 &\:\:\vdots\\
    \nabla^if(t) &=\nabla^{i-1}f(t)-\nabla^{i-1}f(t-\ell)=\sum_{j=0}^i(-1)^j\binom{i}{j}f(t-j\ell).
  \end{align*}
  We may also denote $\nabla^1$ simply by $\nabla$.
\end{definition}

\begin{proposition}[Newton backward difference formula]\label{newton}
  \begin{equation*}
    f(t)=\sum_{j=0}^{n}\binom{n}{j}\nabla^{n-j}f(t-j\ell)
  \end{equation*}
\end{proposition}

\begin{definition}
  Consider a function $f:\ZZ\rightarrow\ZZ$.  If
  \begin{equation*}
    f(t) = \begin{cases}
      f_0(t) &\text{if } t\equiv 0\pmod\ell\\
       & \vdots \\
      f_{\ell-1}(t) &\text{if } t\equiv\ell - 1\pmod\ell,
    \end{cases}
  \end{equation*}
  where each $f_i$ is a polynomial function of degree $d$, then $f$ is
  a \textit{quasipolynomial} function with degree $d$ (denoted $\deg f$)
  and \textit{quasiperiod} $\ell$.
  Further, we will assume that all $f_i$ have a common leading
  coefficient which we denote by $\LC(f)$.

  For brevity, we will refer to a quasipolynomial function with quasiperiod
  $\ell$ as $\ell$-quasipolynomial. An $\ell$-quasipolynomial function of degree 1 is \textit{$\ell$-quasilinear} and one of degree 2 is called \textit{$\ell$-quasiquadratic}, and so on.
\end{definition}

Quasipolynomials are also known as \textit{pseudopolynomials} or \textit{polynomials on residue classes} (PORCs)
\cite{stanley}, and have applications to a wide variety of areas
\cite{woods}.
 
The following useful fact is a result of the power rule from finite calculus.

\begin{proposition}\label{powerrule}
  Suppose $f:\ZZ\rightarrow\ZZ$ is $\ell$-quasipolynomial.  If $\deg f = d$ and $\LC(f)=a$, then $\nabla^df(t) = a\ell^dd!$ and $\nabla^{d+1}f(t)=0$.
\end{proposition}

  The backward difference operator has a nice application to \textit{lattices} of vector spaces. Consider a collection of subspaces of a given vector space. These subspaces generate a \textit{modular lattice} with addition as the join operation, intersection as the meet operation, and $\subseteq$ as partial order.  That is, for every triple $U_1,U_2,U_3$ of subspaces in the lattice with $U_3\subseteq U_1$, we have
  \begin{equation*}
    U_1\cap(U_2+U_3) = (U_1\cap U_2) + (U_1\cap U_3).
  \end{equation*}
However, such a lattice is not \textit{distributive} in the sense that we cannot remove the condition $U_3\subseteq U_1$. Consider, for example, three lines in a plane.
Nevertheless, if we choose our subspaces nicely, then we will have a distributive lattice.  In this case, we may use the inclusion-exclusion principle to compute the dimensions of their sums. See \cite[\S 1.7]{PP} for further discussion of lattices of vector spaces.

Suppose we have a function $N : \NN \to \NN$ and a step size $\ell$.  For every integer $t\geq\ell$, choose an $i\leq \frac{t}{\ell}$. If there exists subspaces $U_1,\ldots, U_i$ of a vector space which generate a distributive lattice and satisfy
\begin{equation*}
  \dim\bigcap_{j\in I}U_j = N(t - |I|\ell)
\end{equation*}
for each $I\subset\{1,\ldots, i\}$, then we call this lattice an $(N,\ell)$-\textit{lattice}.

\begin{lemma}\label{lem_bd_vector_spaces}
  If $U_1,\ldots, U_i$ generate an $(N,\ell)$-lattice, then
  \begin{equation*}
    \dim\sum_{j=1}^i U_j = N(t) - \nabla^iN(t).
  \end{equation*}
\end{lemma}

\begin{proof}
  By inclusion-exclusion, we have
  \begin{align*}
    \dim\sum_{j=1}^iU_j &= \sum_{\emptyset\neq I\subset\{1,\ldots, i\}}(-1)^{|I|-1}\dim\bigcap_{j\in I}U_j\\
                        &= \sum_{j=1}^i(-1)^{j-1}\binom{i}{j}N(t-j\ell) \\
                        &= N(t) -\sum_{j=0}^i(-1)^j\binom{i}{j}N(t-j\ell)\\
                        &= N(t) - \nabla^iN(t);
  \end{align*}
  the last step is by \cref{def_bd}.
\end{proof}

Note that for large enough $i$, we expect the left-hand side of the equation in the statement of \cref{lem_bd_vector_spaces} to be $N(t)$, the dimension of the ambient space, so that $\nabla^i N(t)=0$. Because of \cref{powerrule}, we will be primarily interested in the case where $N$ is $\ell$-quasipolynomial with degree at least $i-1$.

\begin{example}
  If $N(t) = t + 1$, then generic subspaces of $\kk^{t+1}$ of codimension $\ell$ will generate an $(N,\ell)$-lattice.
\end{example}

The following two examples will be especially important for our purposes.

\begin{example}\label{ex_lattice_fixed_deg}
  If $N(t) = \binom{t + d}{d}$ for fixed $d$ and $U_1,\ldots,U_k$ are generic subspaces of $\kk^{t+1}$ of codimension $\ell$ with $k\le\frac{t}{\ell}$, then $S^dU_1,\ldots,S^dU_k$ will generate an $(N,\ell)$-lattice as subspaces of $S^d\kk^{t+1}$.  Indeed, by \cite[Proposition 1.7.1]{PP}, the distributivity of the lattice in $\kk^{t+1}$ is equivalent to the existence of a basis of $\kk^{t+1}$ containing subsets which span each of the $U_j$.  The degree-$d$ monomials generated by these basis vectors provide the basis of $S^d\kk^{t+1}$ needed to show that this second lattice is distributive as well.
\end{example}

\begin{example}\label{ex_lattice_fixed_dim}
  If $N(t) = \binom{n + t}{t}$ for fixed $n$ and $f_1,\dots,f_i\in S^\ell\kk^{n+1}$ are generic, then $f_1S^{t-\ell}\kk^{n+1},\ldots,f_iS^{t-\ell}\kk^{n+1}$ will generate an $(N,\ell)$-lattice as subspaces of $S^t\kk^{n+1}$.
\end{example}

\section{Brambilla--Ottaviani lattices}\label{BO lattices}

In \cite[section 5]{BO}, Brambilla and Ottaviani presented a simplified proof of the Alexander--Hirschowitz theorem \cite{AH} for cubics by specializing points on lattices of linear subspaces. This specific method was adapted to solve similar problems in \cite{abo,AV,wan,thesis}.
In this section, we generalize this technique to any family of projective varieties that admits such lattices.

\subsection{The lattice}
The first step consists of defining a configuration of linear spaces that generalizes the construction from \cite[section 5]{BO}.
Fix some $\ell\in\NN$. Throughout this paper, the backward difference operator $\nabla$ has \textit{step size} $\ell$. Choose $K_0\in\NN$ and define $t_0=\ell K_0+1$. Moreover, for each $t\in\NN$, we set $K(t)=\min\{\lceil t/\ell\rceil - 1,K_0\}$.
Choose an $\ell$-quasipolynomial function $N$ of degree $K_0$ and an $\ell$-quasilinear function $m$.  We say that a family $\{X(t):t\in\NN\}$ of varieties is an $(N, m)$-\textit{family} if $X(t)\subset\PP^{N(t)-1}$ and $\dim X(t)=m(t) - 1$ for all $t\in\NN$.

\begin{definition}\label{def_bo_lattice}
  An $(N,m)$-family $\{X(t):t\in\NN\}$ admits a \textit{Brambilla--Ottaviani lattice} if for every $t>\ell$, there exist ($N(t-\ell)-1)$-planes $P_1(t),\ldots,P_{K(t)}(t)$ in $\PP^{N(t)-1}$ such that
  \begin{enumerate}[(a)]
  \item $\widehat{P}_1(t),\ldots,\widehat{P}_{K(t)}(t)$ generate an $(N,\ell)$-lattice, and
  \item $X(t)\cap\bigcap_{j\in I}P_j(t)\cong X(t-|I|\ell)$ for each nonempty $I\subset\{1,\ldots,K(t)\}$.
  \end{enumerate}
  Furthermore, if $t > t_0$, then there exists an $(N(t_0)-1)$-plane $P'(t)\subset\PP^{N(t)-1}$ such that
  \begin{enumerate}[(a)]
    \setcounter{enumi}{2}
  \item $\widehat{P}'(t)\cap \widehat{P}_1(t),\ldots,\widehat{P}'(t)\cap \widehat{P}_{K(t)}(t)$ generate an $(N,\ell)$-lattice in $\widehat{P}'(t)$, and
  \item $X(t)\cap P'(t)\cap\bigcap_{j\in I}P_j(t)\cong X(t_0-|I|\ell)$ for each $I\subset\{1,\ldots,K_0\}$.
  \end{enumerate}
\end{definition}

Choose an $\ell$-quasipolynomial function $s$ with degree $K_0-1$ and leading coefficient $\LC(N)/\LC(m)$. We place points on $X(t) \subset \PP^{N(t)-1}$ as follows.
\begin{enumerate}[(i)]
 \item If $t \le t_0$, then take $\nabla^{K(t)} s(t)$ generic points in $X(t)$ and for every nonempty subset $I\subset\{1,\ldots,K(t)\}$ pick an additional set of $\nabla^{K(t)-|I|}s(t-|I|\ell)$ generic points in $X(t)\cap\bigcap_{j\in I}P_j(t)$.
 \item If $t>t_0$, then pick $\nabla^{K(t)} s(t)$ generic points in $X(t) \cap P'(t)$ and for every nonempty subset $I\subset\{1,\ldots,K(t)\}$ choose an additional set of
 $\nabla^{K(t)-|I|}s(t-|I|\ell)$ generic points on $X(t)\cap\bigcap_{j\in I}P_j(t)\cap P'(t)$.
\end{enumerate}

Counting the number of points thusly placed, in both cases we find that
\begin{equation*}
  \sum_{j=0}^{K(t)}\binom{K(t)}{j}\nabla^{K(t)-j}s(t-j\ell)=s(t)
\end{equation*}
because of \cref{newton}. The foregoing configuration thus partitions $s(t)$ points. For future reference, we let $Z(t)$ denote the set of all these points for a fixed $t$.

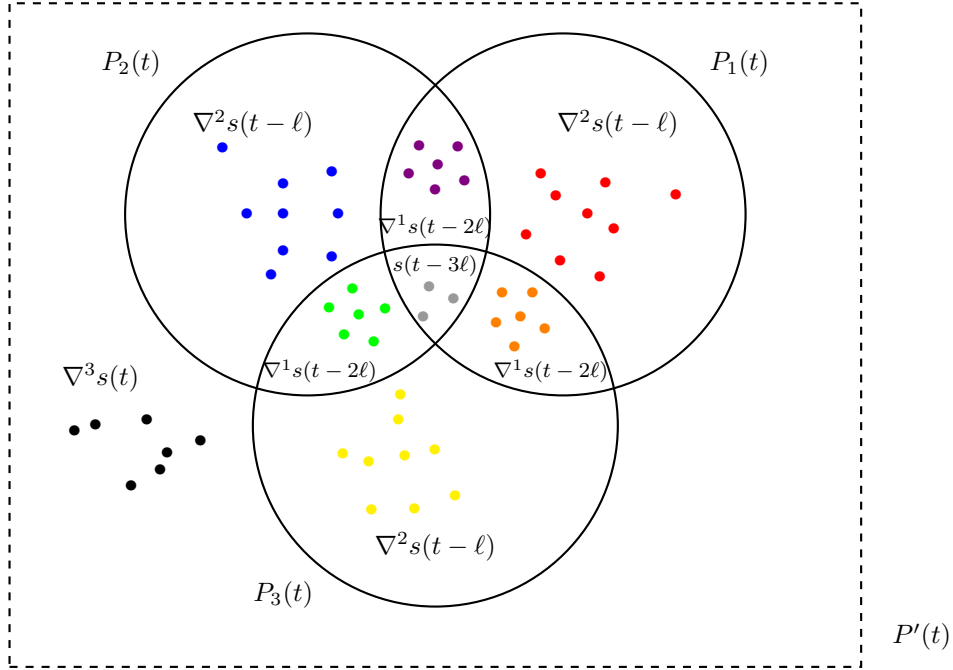
\begin{figure}
\begin{center}
\begin{tikzpicture}[scale=.8]
 \draw[thick,black] (2.1,2.5) circle (3);
 \draw[thick,black] (-2.1,2.5) circle (3);
 \draw[thick,black] (0,-1.0) circle (3);
 \draw[thick,dashed,black] (-7,-5) rectangle (7,6);
\node at (-5.5,-0.2) {$\nabla^3 s(t)$};
\node at (3,4) {$\nabla^{2} s(t - \ell)$};
\node at (-3,4) {$\nabla^{2} s(t - \ell)$};
\node at (0,-3) {$\nabla^{2} s(t - \ell)$};
\node at (0,1.65) {{\footnotesize$s(t-3\ell)$}};
\node at (0,2.3) {{\footnotesize$\nabla^{1} s(t - 2\ell)$}};
\node at (1.9,-.1) {{\footnotesize$\nabla^{1} s(t - 2\ell)$}};
\node at (-1.9,-.1) {{\footnotesize$\nabla^{1} s(t - 2\ell)$}};
\node at (5,5) {$P_1(t)$};
\node at (-5,5) {$P_2(t)$};
\node at (-2.5,-3.8) {$P_3(t)$};
\node at (8,-4.5) {$P'(t)$};
\begin{scope}[shift={(2.5,2.5)},rotate=60,color=red]
\node at (-.8,-.7) {$\bullet$};
\node at (0,-.5) {$\bullet$};
\node at (1,-1.1) {$\bullet$};
\node at (-0.9,0) {$\bullet$};
\node at (0,0) {$\bullet$};
\node at (0.6,0) {$\bullet$};
\node at (-.8,.7) {$\bullet$};
\node at (0,.6) {$\bullet$};
\node at (0.2,1) {$\bullet$};
\end{scope}
\begin{scope}[shift={(-2.5,2.5)},rotate=180,color=blue]
\node at (-.8,-.7) {$\bullet$};
\node at (0,-.5) {$\bullet$};
\node at (1,-1.1) {$\bullet$};
\node at (-0.9,0) {$\bullet$};
\node at (0,0) {$\bullet$};
\node at (0.6,0) {$\bullet$};
\node at (-.8,.7) {$\bullet$};
\node at (0,.6) {$\bullet$};
\node at (0.2,1) {$\bullet$};
\end{scope}
\begin{scope}[shift={(-0.5,-1.5)},rotate=100,color=yellow]
\node at (-.8,-.7) {$\bullet$};
\node at (0,-.5) {$\bullet$};
\node at (1,-0.1) {$\bullet$};
\node at (-0.9,0) {$\bullet$};
\node at (0,0) {$\bullet$};
\node at (0.6,0) {$\bullet$};
\node at (-.8,.7) {$\bullet$};
\node at (0,.6) {$\bullet$};
\node at (0.2,1) {$\bullet$};
\end{scope}
\begin{scope}[shift={(1,0.7)},rotate=0,color=orange]
\node at (0,0) {$\bullet$};
\node at (0.1,0.5) {$\bullet$};
\node at (0.3,-0.4) {$\bullet$};
\node at (0.4,0.1) {$\bullet$};
\node at (0.6,0.5) {$\bullet$};
\node at (0.8,-0.1) {$\bullet$};
\end{scope}
\begin{scope}[shift={(-1.5,0.5)},rotate=40,color=green]
\node at (0,0) {$\bullet$};
\node at (0.1,0.5) {$\bullet$};
\node at (0.3,-0.4) {$\bullet$};
\node at (0.4,0.1) {$\bullet$};
\node at (0.6,0.5) {$\bullet$};
\node at (0.8,-0.1) {$\bullet$};
\end{scope}
\begin{scope}[shift={(0,2.9)},rotate=70,color=violet]
\node at (0,0) {$\bullet$};
\node at (0.1,0.5) {$\bullet$};
\node at (0.3,-0.4) {$\bullet$};
\node at (0.4,0.1) {$\bullet$};
\node at (0.6,0.5) {$\bullet$};
\node at (0.8,-0.1) {$\bullet$};
\end{scope}
\begin{scope}[shift={(-.5,1.2)},rotate=0,color=black!40]
\node at (0.3,-0.4) {$\bullet$};
\node at (0.4,0.1) {$\bullet$};
\node at (0.8,-0.1) {$\bullet$};
\end{scope}
\begin{scope}[shift={(-5,-2)},rotate=50]
\node at (1.3,-0.4) {$\bullet$};
\node at (0.4,1.1) {$\bullet$};
\node at (0.8,-0.1) {$\bullet$};
\node at (0.5,-0.2) {$\bullet$};
\node at (0.1,1.3) {$\bullet$};
\node at (0.0,0.0) {$\bullet$};
\node at (1.0,0.5) {$\bullet$};
\end{scope}

\end{tikzpicture}
\end{center}
\caption{Illustration of the point configuration with $K(t)=3$. The number of points is $s(t) = \binom{3}{0} \nabla^3 s(t) + \binom{3}{1} \nabla^2(s - \ell) + \binom{3}{2} \nabla s(t - 2\ell) + \binom{3}{3} s(t - 3\ell)$. The first term corresponds to the black points, the second term to the red, blue and yellow points, the third term to the violet, green and orange points, and the last term to the gray points. If $t \le t_0$, then the dashed rectangle representing $P'(t)$ is not present.}
\end{figure}

For all $t$ and all $i\in\{0,\ldots,K(t)\}$, we introduce the vector space
\begin{equation} \label{eqn_original_A}
  A_i(t)=\sum_{j=1}^i\widehat{P}_j(t) + \sum_{p\in Z(t)} T_p \widehat{X}(t),
\end{equation}
which will be the main focus of this paper. The main lemma of this paper is that
\begin{equation*}
  a_i(t) = N(t) - \nabla^i N(t) + i \nabla m(t) \cdot \nabla^{i-1}s(t-\ell) + m(t) \cdot \nabla^i s(t).
\end{equation*}
constitutes an upper bound on the dimension of $A_i(t)$.
\begin{lemma}\label{lem_main_lemma}
  For all $t\in\NN$, $\dim A_i(t)\leq a_i(t)$.
\end{lemma}
\begin{proof}
  Consider a point $[p] \in Z(t)\cap P_j(t)\cap P_k(t)$ for distinct $j,k\in\{1,\ldots,K(t)\}$ with $K(t)\geq 2$.  By \cref{def_bo_lattice}(b), $T_p \widehat{X}(t)$ intersects $\widehat{P}_j(t)$ and $\widehat{P}_k(t)$ each in dimension $m(t-\ell)$ and $\widehat{P}_j(t) \cap \widehat{P}_k(t)$ in dimension $m(t-2\ell)$.  Since $m$ is $\ell$-quasilinear, $\nabla^2m(t)=0$ by \cref{powerrule}.  It follows that
  {\small\begin{align*}
    \dim(T_p \widehat{X}(t)\cap(\widehat{P}_j(t)+\widehat{P}_k(t)))
    &=2m(t-\ell)-m(t-2\ell)\\
    &=m(t) - ( m(t) - m(t-\ell) ) + ( m(t-\ell) - m(t-2\ell) )\\
    &=m(t)-\nabla^2m(t) =m(t),
  \end{align*}}%
  i.e., $T_p \widehat{X}(t) \subset\widehat{P}_j(t)+\widehat{P}_k(t)$.
  So we may remove some unnecessary summands from the definition of $A_i(t)$. Specifically, we obtain
  \begin{equation} \label{eqn_simple_A}
    A_i(t) = \sum_{j=1}^i \widehat{P}_j(t) + \sum_{j=1}^i \sum_{[p]\in Z_{i,j}(t)} T_p \widehat{X}(t) +\sum_{[p]\in Z_i(t)} T_p \widehat{X}(t),
  \end{equation}
  where
  \begin{align*}
    Z_i(t)&=Z(t)\setminus(P_1(t)\cup\cdots\cup P_i(t)),\\
    Z_{i,j}(t)&= Z(t) \cap \left(P_j(t) \setminus(P_1(t)\cup\cdots\cup P_{j-1}(t)\cup P_{j+1}(t)\cup\cdots\cup P_i(t)) \right).
  \end{align*}
  Note that for $K(t)=1$, we have $Z_1(t)=Z(t)\setminus P_1(t)$ and $Z_{1,1}(t) = Z(t) \cap P_1(t)$ so that $Z_1(t) \cup Z_{1,1}(t)=Z(t)$ and the above decomposition of $A_i(t)$ holds as well.

  Next, we use \cref{eqn_simple_A} to bound the dimension. The first sum contributes
  \begin{equation*}
    \dim\sum_{j=1}^i\widehat{P}_j(t)=N(t)-\nabla^iN(t)
  \end{equation*}
  to $\dim A_i(t)$ because of \cref{lem_bd_vector_spaces}.

  The contribution of the second sum is determined next.
  By \cref{def_bo_lattice}(b), $T_p \widehat{X}(t)$ intersects $\widehat{P}_j(t)$ in dimension $m(t-\ell)$ for each $[p]\in Z_{i,j}(t)$. Consequently, $T_p \widehat{X}(t)$ modulo $\widehat{P}_j(t)$ adds at most $m(t)-m(t-\ell)=\nabla m(t)$ to the dimension of $A_i(t)$. The points in $Z_{i,j}(t)$ correspond to those subsets of $\{1,\ldots,K(t)\}$ that can be expressed as $I = \{j\} \cup J$ where $J$ is a subset of $\{i+1,\ldots,K(t)\}$. That is,
  \[
   Z_{i,j}(t) = \bigcup_{J \subset \{i+1,\ldots,K(t)\}} \left( Z(t) \cap \bigcap_{k \in J \cup \{j\}} P_k(t) \right).
  \]
  By the definition of the configuration of points $Z(t)$, we see that $Z_{i,j}(t)$ contains
  \begin{dmath*}
    \sum_{k=0}^{K(t)-i}\binom{K(t)-i}{k}\nabla^{K(t)-(k+1)}s(t-(k+1)\ell)
    =\sum_{k=0}^{K(t)-i}\binom{K(t)-i}{k}\nabla^{K(t)-i-k}(\nabla^{i-1}s(t-(k+1)\ell))
    =\nabla^{i-1}s(t-\ell)
  \end{dmath*}
  points, where the last equality is due to \cref{newton} and $\nabla^\alpha \nabla^\beta = \nabla^{\alpha+\beta}$ for all $\alpha, \beta \in \NN$. It follows that the second sum in \cref{eqn_simple_A} adds at most $i \nabla m(t) \cdot \nabla^{i-1} s(t-\ell)$ to the dimension of $A_i(t)$, modulo the first sum in \cref{eqn_simple_A}.

Finally, we see that $Z_i(t) = \cup_{J \subset \{i+1,\ldots,K(t)\}} Z(t) \cap (\cap_{j\in J} P_j(t))$, so that it contains
  \begin{dmath*}
    \sum_{j=0}^{K(t)-i}\binom{K(t)-i}{j}\nabla^{K(t)-j}s(t-j\ell)
    =\sum_{j=0}^{K(t)-i}\binom{K(t)-i}{j}\nabla^{K(t)-i-j}(\nabla^is(t-j\ell))
    =\nabla^i s(t)
  \end{dmath*}
  points, where the last step is due to \cref{newton}. The last sum in \cref{eqn_simple_A} thus contributes at most $m(t) \cdot \nabla^i s(t)$ to the dimension of $A_i(t)$.
\end{proof}

\subsection{The induction}
Since $A_i(t)$ is constructed by picking points generically, we expect that, unless it fills the ambient space $\kk^{N(t)}$, its dimension will attain the upper bound from \cref{lem_main_lemma}. For this reason we introduce the next definition.
\begin{definition}\label{defn_expdim}
The \textit{expected dimension} of $A_i (t)$ is
\[
\operatorname{expdim} A_i(t) := \min\{a_i(t),N(t)\}.
\]
We say that the statement $\statement{A}_i(t)$ is \textit{true} if $\dim A_i(t) = \operatorname{expdim} A_i(t)$ and \textit{false} otherwise.
The statement $\statement{A}_i(t)$ is called \textit{subabundant} if $a_i(t)\leq N(t)$, \textit{superabundant} if $a_i(t)\geq N(t)$, and \textit{equiabundant} if $a_i(t)=N(t)$.
If $s$ is not clear from context, then we write $\statement{A}_{i,s}(t)$.
\end{definition}

The main reason for studying $A_i(t)$ is the following lower bound for $i=0$:
\begin{equation} \label{eqn_lower_bound}
  \dim A_0(t) \le \dim\sigma_{s(t)}X(t) + 1
\end{equation}
which follows from \cref{eqn_original_A}, \cref{terracini}, and semicontinuity. Therefore, if 
\[
\dim A_0(t) = \operatorname{expdim} A_0(t) = \min\{ s(t)m(t), N(t) \} = \expdim\sigma_{s(t)}X(t) + 1, 
\]
then $\sigma_{s(t)}X(t)$ is nondefective. In other words, if the statement $\statement{A}_0(t)$ is true, then the secant variety $\sigma_{s(t)}(X(t))$ is nondefective.

In the next series of lemmata we develop the induction strategy.

\begin{lemma}\label{induction1}
  Suppose $i<K(t)$ and $t>\ell$. If $\statement{A}_i(t-\ell)$ and $\statement{A}_{i+1}(t)$ are both true and subabundant (respectively superabundant), then $\statement{A}_i(t)$ is true and subabundant (respectively superabundant).
\end{lemma}

\begin{proof}
  By construction, $A_i(t)\cap\widehat{P}_{i+1}(t) \cong A_i(t-\ell)$
  and $A_i(t)+\widehat{P}_{i+1}(t)=A_{i+1}(t)$, and therefore, by Grassmann's formula,
  \begin{equation*}
    \dim A_i(t) = \dim A_i(t-\ell)+\dim A_{i+1}(t)-N(t-\ell).
  \end{equation*}

  \paragraph{\textit{Case 1.}}
  Suppose $\statement{A}_i(t-\ell)$ and $\statement{A}_{i+1}(t)$ are both true and subabundant.  Then
  \begin{dmath*}
    \dim A_i(t) 
    = a_i(t-\ell)+a_{i+1}(t)-N(t-\ell)
    = N(t-\ell)-\nabla^iN(t-\ell)+i \nabla m(t-\ell) \cdot \nabla^{i-1}s(t-2\ell)  + m(t-\ell) \cdot \nabla^is(t-\ell)  + 
    N(t)-\nabla^{i+1}N(t)+(i+1) \nabla m(t) \cdot \nabla^is(t-\ell) + m(t) \cdot \nabla^{i+1}s(t) - N(t-\ell)
    = N(t) - \underbrace{\left( \nabla^iN(t-\ell) + \nabla^{i+1}N(t) \right)}_{(a)} + \underbrace{\left( i \nabla m(t-\ell) \cdot \nabla^{i-1}s(t-2\ell)+(i+1) \nabla m(t) \cdot \nabla^is(t-\ell) \right)}_{(b)} + \underbrace{\left(m(t-\ell) \cdot \nabla^is(t-\ell) + m(t) \cdot \nabla^{i+1}s(t) \right)}_{(c)}
  \end{dmath*}
Now observe that $(a) = \nabla^i N(t-\ell) + \nabla^i N(t) - \nabla^iN(t-\ell) = \nabla^i N(t)$. For the next term we find
\begin{dmath*}
   (b) = i \left( \nabla m(t-\ell) \cdot \nabla^{i-1}s(t-2\ell)  + \nabla m(t) \left( \nabla^{i-1} s(t-\ell) - \nabla^{i-1} s(t-2\ell) \right) \right)
   +\nabla m(t) \cdot \nabla^is(t-\ell)
   = i \nabla m(t) \cdot \nabla^{i-1} s(t-\ell) - i \nabla^2 m(t) \nabla^{i-1} s(t-2\ell) + \nabla m(t) \cdot \nabla^is(t-\ell)
   = i \nabla m(t) \cdot \nabla^{i-1} s(t-\ell) + \nabla m(t) \cdot \nabla^is(t-\ell),
\end{dmath*}
because $m$ is $\ell$-quasilinear, so that $\nabla^2 m(t) = 0$ by \cref{powerrule}. For $(c)$, it suffices to note that
\begin{dmath*}
 (c) = m(t-\ell) \cdot \nabla^i s(t-\ell) + m(t) \left( \nabla^i s(t) - \nabla^i s(t-\ell) \right) 
 = m(t) \cdot \nabla^i s(t) - \nabla m(t) \cdot \nabla^i s(t-\ell).
\end{dmath*}
Putting everything together, we find
  \begin{dmath*}
  \dim A_i(t)
    \hiderel{=} N(t) - \nabla^i N(t) + i \nabla m(t) \cdot \nabla^{i-1}s(t-\ell) + m(t) \cdot \nabla^i s(t)
    = a_i(t).
  \end{dmath*}

\paragraph{\textit{Case 2.}} 
Suppose $\statement{A}_i(t-\ell)$ and $\statement{A}_{i+1}(t)$ are both true and superabundant. Then
  \begin{equation*}
    \dim A_i(t) = N(t-\ell) + N(t) - N(t-\ell)=N(t),
  \end{equation*}
  concluding the proof.
\end{proof}

\begin{lemma}\label{equiabundant}
  $\statement{A}_{K_0}(t)$ is equiabundant for all $t\geq t_0$.
\end{lemma}
\begin{proof}
  Using \cref{powerrule},
  \begin{dmath*}
    a_{K_0}(t) = N(t) - \nabla^{K_0}N(t)+ K_0 \nabla m(t) \cdot \nabla^{K_0-1} s(t-\ell) + m(t) \cdot \nabla^{K_0}s(t)
               = N(t)-K_0! \LC(N) \ell^{K_0} + K_0 \LC(m)\ell \cdot \LC(s)(K_0-1)!\ell^{K_0-1}+0
               = N(t)-K_0! \LC(N)\ell^{K_0}+\frac{\LC(N)}{\LC(m)}\LC(m)K_0!\ell^{K_0}
               = N(t).
  \end{dmath*}
This concludes the proof.
\end{proof}

\begin{lemma}\label{induction2}
  If $\statement{A}_{K_0}(t_0)$ is true, then $\statement{A}_{K_0}(t)$ is true for all $t\geq t_0$.
\end{lemma}

\begin{proof}
  By construction, $A_{K_0}(t)\cap\widehat{P}'(t) \cong A_{K_0}(t_0)$,
  and therefore,
  \begin{align*}
    \dim A_{K_0}(t) &= \dim(A_{K_0}(t)+\widehat{P}'(t))+\dim A_{K_0}(t_0)-N(t_0)\\
                    &= \dim(A_{K_0}(t)+\widehat{P}'(t))\\
                    &= \dim\left(\sum_{j=1}^{K_0}\left(\widehat{P}_j(t)+\sum_{p\in Z_{K_0,j}(t)} T_p \widehat{X}(t) \right)+\sum_{p\in Z_{K_0}(t)}T_p \widehat{X}(t) + \widehat{P}'(t)\right).
  \end{align*}

  Note that $Z_{K_0}(t)=\emptyset$ since $\nabla^{K_0}s(t)=0$ by \cref{powerrule}.
  By construction, if $[p]\in Z_{K_0,j}(t)$, then $T_p \widehat{X}(t)$ intersects $\widehat{P}_j(t)$ in dimension $m(t-\ell)$, $\widehat{P}'(t)$ in dimension $m(t_0)$, and $\widehat{P}_j(t) \cap \widehat{P}'(t)$ in dimension $m(t_0-\ell)$. Therefore, $T_p \widehat{X}(t)$ intersects $\widehat{P}_j(t) + \widehat{P}'(t)$ in dimension 
  \[
  m(t-\ell)+m(t_0)-m(t_0-\ell)=m(t)-\nabla m(t)+\nabla m(t_0)= m(t),
  \]
  where in the last equality we used that $m$ is $\ell$-quasilinear, so $\nabla m(t)=\nabla m(t_0)$ by \cref{powerrule}. Since $m(t)$ is the dimension of $T_p \widehat{X}(t)$, we conclude that $T_p \widehat{X}(t)\subset\widehat{P}_j(t)+\widehat{P}'(t)$. Putting everything together, we obtain
  \begin{align*}
    \dim A_{K_0}(t) &= \dim\left(\sum_{j=1}^{K_0}\widehat{P}_j(t)+\widehat{P}'(t)\right)\\
                    &= N(t)-\nabla^{K_0}N(t)+N(t_0) - (N(t_0)-\nabla^{K_0}N(t_0))\\
                    &= N(t) -\nabla^{K_0}N(t) + \nabla^{K_0}N(t_0),
  \end{align*}
  where the second step is by \cref{def_bo_lattice}.
  Finally, we recall that the degree of $N$ is $K_0$ by definition and \cref{powerrule} entails that $\nabla^{K_0}N(t)=\nabla^{K_0}N(t_0)$.  Therefore, $\dim A_{K_0}(t)=N(t)$, the expected dimension by \cref{equiabundant}.
\end{proof}

Now we are ready to combine the foregoing lemmata into the main theorem.

\begin{theorem}\label{base cases}
  If $\statement{A}_{K(t),s}(t)$ is true and subabundant (respectively superabundant) for all $t\leq t_0$, then $\statement{A}_{i,s}(t)$ is true and subabundant (respectively superabundant) for all $t\in\NN$ and $i\in\{0,\ldots,K(t)\}$.  In particular, $\sigma_{s(t)}(X(t))$ is nondefective for all $t$.
\end{theorem}
\begin{proof}
  Since $\statement{A}_{K_0,s}(t_0)$ is true by assumption, $\statement{A}_{K_0,s}(t)$ is also true for $t \ge t_0$ by \cref{induction2}. These together with the other assumptions form the base cases.
  The remaining cases follow immediately using induction on $i$ and $t$ with step size $\ell$ by \cref{induction1}.
  Under the assumptions of the theorem, $\statement{A}_0(t)$ is true, so it follows from \cref{eqn_lower_bound} that $\sigma_{s(t)} X(t)$ is nondefective for all $t \in \NN$.
\end{proof}

\begin{corollary}\label{cor_nondefective}
  Suppose $s_2(t)=\left\lceil\frac{N(t)}{m(t)}\right\rceil$ is $\ell$-quasipolynomial.  Then so is $s_1(t)=s_2(t)-1$.  If $\statement{A}_{K(t),s_1}(t)$ is true and subabundant and $\statement{A}_{K(t),s_2}(t)$ is true and superabundant for all $t\leq t_0$, then $\sigma_s(X(t))$ is nondefective for all $s,t\in\NN$.
\end{corollary}
\begin{proof}
  The assumptions entail that $\sigma_{s_1(t)}(X(t))$ and $\sigma_{s_2(t)}(X(t))$ are both nondefective due to \cref{base cases}. 
  If $s<s_1(t)$, then there are fewer summands when computing the dimension of $\sigma_s(X(t))$ using Terracini's Lemma (\cref{terracini}) than there are for $\sigma_{s_1(t)}(X(t))$. As the latter has the expected dimension, so must the former. A similar argument holds for $s > s_2(t)$.
\end{proof}

Equivalently we can ask for $s_1(t) = \left\lfloor \frac{N(t)}{m(t)} \right\rfloor$ to be $\ell$-quasilinear with $s_2(t) = s_1(t)+1$ in the previous result.

\subsection{Examples}
The foregoing technique generalizes several instances that have appeared in the literature. We review some of them next.

\begin{example}[Secant varieties of third Veronese varieties \cite{BO}]
  Brambilla and Ottaviani's original result was presented in the different but, thanks to inverse systems (see \cite{iarrobino}), equivalent language of determining the number of double points that impose independent conditions on cubics.

  Let $X(t)$ be the third veronese embedding of $\PP^{t+4}$, i.e., $X(t) = v_3(\PP^{t+4})$, so that $N(t)=\binom{t+7}{3}$ and $m(t)=t+5$. Let $\ell=K_0=3$ and $t_0=10$. We begin with senary (six variable) cubics, as the quinary case is well known to be defective.
  In this example, $N(t)/m(t)=\frac{1}{6}(t+7)(t+6)$. If $t\equiv 0,2\pmod 3$, this is always an integer, and for $t\equiv 1\pmod 3$, this is always exactly $\frac{2}{3}$ less than an integer.  Therefore, $s_2(t) = \left\lceil \frac{N(t)}{m(t)} \right\rceil$ is $3$-quasiquadratic.

  Note that $X(t)\subset\PP(S^3V(t))$ for some $(t+5)$-dimensional vector space $V(t)$. For each $t > \ell = 3$ and $j\in\{1,\ldots, K(t)\}$, choose a generic $(t+2)$-dimensional subspace $U_j(t)$ of $V(t)$, forming an $(N,3)$-lattice as in \cref{ex_lattice_fixed_deg}, and let $P_j(t)=\PP(S^3U_j(t))$.  For each $t>10$, let $U'(t)$ be a 15-dimensional subspace of $V(t)$ and define $P'(t)=\PP(S^3U'(t))$.  These form our Brambilla--Ottaviani lattice.


  We can use a computer to construct and find the dimensions of the vector spaces $A_{K(t)}(t)$ with $t\leq 10$ to verify the base cases needed for \cref{cor_nondefective}.
  It follows that $\sigma_s(v_3(\PP^n))$ is nondefective for all $n\geq 5$.
\end{example}

\begin{example}[Secant varieties of tangential varieties to third Veronese varieties \cite{AV}] \
  Let $X(t)$ be the tangential variety of the third veronese embedding of $\PP^{t+7}$, i.e., $X(t) =\tau(v_3(\PP^{t+7}))$. Then, $N(t)=\binom{t+10}{3}$ and $m(t)=2t+15$. Let $\ell=24$, $K_0=3$, and $t_0=73$.  In this case, $s_2(t) = \left\lceil \frac{N(t)}{m(t)} \right\rceil$ is $24$-quasiquadratic. The Brambilla--Ottaviani lattice is formed similarly as in the previous example.


  The base cases needed for \cref{cor_nondefective} may be computed using software, although due to the step size required by the induction, this is a considerably more difficult task than the previous example, requiring computations in a space of dimension $88\,560$. It nevertheless follows from the computations described in \cite{AV} that $\sigma_s(\tau(v_3(\PP^n)))$ is nondefective for all $n\geq 8$. Note that they started the induction only at $n=8$ to avoid some numerical issues that arose when starting with smaller $n$.
  The results for $n\leq 7$ were previously known \cite{BCGI}, hence completing the proof.
\end{example}

\begin{example}[Secant varieties of Segre-Grassmann varieties \cite{wan}]
  Let $X(t)$ be the Segre embedding of $\PP^t$ and the (Pl\"ucker embedding of the) Grassmannian of $1$-dimensional subspaces of $\kk^{t+3}$, i.e.,
  $X(t)=\Seg(\PP^t\times\GG(1, t+2))$, so $N(t)=(t+1)\binom{t+2}{2}$ and $m(t)=3t+5$. Let $\ell=6$, $K_0=3$, and $t_0=19$. In this example, $s_2(t) = \left\lceil \frac{N(t)}{m(t)} \right\rceil$ is 6-quasiquadratic.

  Note that $X(t)\subset\PP\left(V_1(t)\otimes \bigwedge^2 V_2(t)\right)$ for some $(t+1)$-dimensional vector space $V_1(t)$ and $(t+3)$-dimensional vector space $V_2(t)$. For every $t > \ell = 6$ and $j\in\{1,\ldots, K(t)\}$, choose a generic $(t-5)$-dimensional subspace $U_{1,j}(t)$ of $V_1(t)$, a generic $(t-3)$-dimensional subspace $U_{2,j}(t)$ of $V_2(t)$, and take $P_j(t)=\PP\left(U_{1,j}(t)\otimes\bigwedge^2 U_{2,j}(t)\right)$. For every $t > t_0 = 19$, let $U_1'(t)$ be some 20-dimensional subspace of $V_1(t)$, $U_2'(t)$ a 22-dimensional subspace of $V_2(t)$, and define $P'(t)=\PP\left(U_1'(t)\otimes\bigwedge^2U_2'(t)\right)$. These form the Brambilla--Ottaviani lattice used in \cite{wan}.


  The base cases needed for \cref{cor_nondefective} can be verified by computer, in this way proving that $\sigma_s(\Seg(\PP^n\times\GG(1,n+2)))$ is nondefective for all $n\geq 1$.

  Similar constructions were used by Wan \cite{wan} to prove the nondefectivity of $\sigma_s(\Seg(\PP^n\times\GG(1,n-4)$ for $n=5$, $n=6$, and $n\geq 10$ and of $\sigma_s(\Seg(\PP^n\times\GG(1,n-1)))$ for $n=2$ and $n\geq 6$.
\end{example}

\section{Chow varieties}
\label{chow BO lattices}

\subsection{Definitions}

Let $\kk$ be an algebraically closed field of characteristic 0 and suppose $V$ is an $(n+1)$-dimensional vector space over $\kk$.  Then the $d$th symmetric power $S^dV$ is the $\binom{n+d}{d}$-dimensional vector space of $(n+1)$-ary $d$-ics, i.e., homogeneous polynomials of degree $d$ in $n+1$ variables with coefficients in $\kk$.

\begin{definition}
  The \textit{Chow variety} (also known as the \textit{split variety} or \textit{variety of completely decomposable or reducible forms}) of $(n+1)$-ary $d$-ics is the projective variety
  \begin{equation*}
   \mathcal{C}_{d,n} = \Split_d(\PP^n) = \{[\ell_1\cdots\ell_d]:\ell_i\in V\}
  \end{equation*}
  in $\PP V=\PP^{\binom{n+d}{d}-1}$.
\end{definition}

By the product rule, at $[p]=[\ell_1\cdots\ell_d]$,
\begin{equation}\label{eqn_product_rule_tangent_space}
  T_p \widehat{\mathcal{C}}_{d,n}=\sum_{i=1}^d\ell_1\cdots\ell_{i-1}\ell_{i+1}\cdots\ell_dV.
\end{equation}
If $p$, and thus the $\ell_i$, are generic, then the pairwise intersection of the summands in \cref{eqn_product_rule_tangent_space} is the line spanned by $p$.  Since each summand has dimension $n+1$, we may inductively apply Grassman's formula to see that $\dim T_p\widehat{\mathcal{C}}_{d,n}=dn + 1$.  It follows that $\dim\mathcal{C}_{d,n}=dn$, and consequently, from \cref{def_expdim}, we have
\begin{equation*}
  \expdim\sigma_s(\mathcal{C}_{d,n})=\min\left\{s(dn+1),\binom{n+d}{d}\right\}-1.
\end{equation*}

Several of the known results concerning dimensions of secant varieties of Chow varieties have arisen via Brambilla--Ottaviani lattices.  In fact, we may use them to perform induction on both dimension ($n$) and degree ($d$). Since $\binom{n+d}{d}$ and $dn+1$ are both symmetric in these variables, these processes
are very similar.

\subsection{Induction on degree with fixed dimension} \label{sec_ind_degree_chow}

Fixing $n$, we may let $X(t)=\mathcal{C}_{t+\alpha,n}$ (where $d=\alpha+1$ is the starting point of our induction)$, N(t)=\binom{n + t+\alpha}{t+\alpha}$, $m(t)=(t+\alpha)n+1$, and $K_0=n$.  The values of $\ell$ (and thus $t_0$) are generally chosen to be large enough so that $s(t)\approx \frac{N(t)}{m(t)}$ for all $t$ but also small enough so that $N(t_0)$ (the dimension of the largest vector space in which we will doing computations) is not unmanageably large.

Note that $X(t)\subset\PP(S^{t+\alpha}V)$ for some $(n+1)$-dimensional vector space $V$.  For each $t>\ell$ and $j\in\{1,\ldots,K(t)\}$, choose $\ell$ generic linear forms $g_{j,1},\ldots,g_{j,\ell}\in V$ and let $P_j(t)=\PP(g_{j,1}\cdots g_{j,\ell}S^{t+\alpha-\ell} V)$ as in \cref{ex_lattice_fixed_dim}. For each $t>t_0$, choose $t-t_0$ generic linear forms $h_1,\ldots,h_{t-t_0}\in V$ and define $P'(t)=\PP(h_1\cdots h_{t-(t_0+\alpha)}S^{t_0+\alpha}V)$. These form our Brambilla--Ottaviani lattice.

\subsubsection{Ternary forms}

In \cite{abo}, Abo considered the $n=2$ case using $\alpha=5$, $\ell=4$, and $t_0=9$.  In this case,
\begin{equation*}
  \left\lceil\frac{N(t)}{m(t)}\right\rceil=
  \frac{1}{4}\cdot\begin{cases}
    t + 8 & \text{if }t\equiv 0\pmod 4\\
    t + 11 & \text{if }t\equiv 1\pmod 4\\
    t + 10 & \text{if }t\equiv 2\pmod 4\\
    t + 9 & \text{if }t\equiv 3\pmod 4
  \end{cases}
\end{equation*}

This is 4-quasilinear, and after verifying the base cases from Corollary \ref{cor_nondefective}, we see that $\sigma_s(\mathcal{C}_{d,2})$ is nondefective for all $d\geq 6$.  As Shin had already proven the result for $d\leq 5$ in \cite{shin1}, this completes the proof of Conjecture \ref{conjecture} for $n=2$.

\subsubsection{Quaternary forms}\label{old n=3 case}

Abo also considered the $n=3$ case using $\alpha=0$, $\ell=6$, and $t_0=19$ with
a pair of 6-quasiquadratic functions $s_1$ and $s_2$ satisfying $s_1(t)\leq \frac{N(t)}{m(t)}\leq s_2(t)$ for all $t$. However, $s_2(t)>\left\lceil\frac{N(t)}{m(t)}\right\rceil$ for $t>6$ and $s_2-s_1$ is 6-quasilinear. So they are not sufficient for completing the $n=3$ case using Corollary \ref{cor_nondefective}.  However, checking the base cases needed for Theorem \ref{base cases} does show that $\sigma_s(\mathcal{C}_{t,3})$ is nondefective for all $s\leq s_1(t)$ and $s\geq s_2(t)$, leaving gap of unknown cases near $\frac{N(t)}{m(t)}$ which widens as $t$ grows.

In his Ph.D. thesis \cite{thesis}, the first author improved on this slightly using $\alpha=0$, $\ell=9$, and $t_0=28$ with a pair of 9-quasiquadratic functions $s'_1$ and $s'_2$ that satisfy $s'_1(t)\geq s_1(t)$ for all $t$ and $s'_2(t)<s_2(t)$ for all $t\not\equiv 0\pmod 3$. This closes the gap slightly for completing the $n=3$ case of Conjecture \ref{conjecture}, but a gap still remains.

\subsection{Induction on dimension with fixed degree} \label{sec_ind_dimension_chow}
On the other hand, if we fix $d$, we may let $X(t)=\mathcal{C}_{d,t+\alpha}$ (where $n=\alpha+1$ is the starting point of our induction)$, N(t)=\binom{t+\alpha+d}{d}$, $m(t)=d(t+\alpha)+1$, $K_0=d$, and values of $\ell$ and $t_0$ are chosen depending on the specific problem.

Note that $X(t)\subset\PP(S^dV(t))$ for some $(t+\alpha+1)$-dimensional vector space $V(t)$.  For each $t>\ell$ and $j\in\{1,\ldots,K(t)$, choose a generic $(t-\ell+1)$-dimensional vector subspace $U_j(t)$ of $V(t)$ and let $P_j(t)=\PP(S^d U_j(t))$.  For each $t>t_0$, let $U'(t)$ be an $(t_0+\alpha+1)$-dimensional subspace of $V(t)$ and define $P'(t)=\PP(S^dU'(t))$. These form our Brambilla--Ottaviani lattice.

\subsubsection{Cubics}\label{old d=3 case}

In \cite{abo}, Abo considered the $d=3$ case using $\alpha=0$, $\ell=6$, and $t_0=19$ using the same 6-quasiquadratic functions $s_1$ and $s_2$ as the $n=3$ case, thanks to the symmetry of $\binom{n+d}{d}$ and $nd+1$ in $n$ and $d$. Checking the base cases from Theorem \ref{base cases} proves that $\sigma_s(\mathcal{C}_{3,n})$ is nondefective for all $s\leq s_1(t)$ and $s\geq s_2(t)$.

The functions $s'_1$ and $s'_2$ from \cite{thesis} could theoretically be used
to close the gap somewhat as in the $n=3$ case, but the first author did not
have access to the computing power necessary to verify the necessary base cases.


\section{Proof of \texorpdfstring{\cref{main_result}}{Theorem \ref{main_result}}} \label{verification}

To prove \cref{main_result} for an algebraically closed field $\kk$ of characteristic $0$ that contains $\NN$, we select the $27$-quasiquadratic functions that the first author identified in \cite{thesis}, namely
\begin{equation} \label{eqn_the_esses}
\begin{aligned}
  s_1(t) &= \frac{1}{18}t^2 + \frac{17}{54}t + \frac{a(t)}{27}\\
  s_2(t) &= s_1(t) + 1,
\end{aligned}
\end{equation}
where $a(t)$ depends on the remainder of $t$ when divided by 27 as given in
Table \ref{values of a}.

\begin{table}[ht]
\centering
\begin{equation*}
\begin{array}{|r|r||r|r||r|r|}
\hline
r & a(27q+r) & r & a(27q+r) & r & a(27q+r) \\\hline
0 & 0   & 9  & -9  & 18 & 9 \\
1 & -10 & 10 & 8   & 19 & -1 \\
2 & 4   & 11 & -5  & 20 & 13 \\
3 & -12 & 12 & 6   & 21 & -3 \\
4 & -4  & 13 & -13 & 22 & 5 \\
5 & 1   & 14 & -8  & 23 & 10 \\
6 & 3   & 15 & -6  & 24 & 12 \\
7 & 2   & 16 & -7  & 25 & 11 \\
8 & -2  & 17 & -11 & 26 & 7 \\
\hline
\end{array}
\end{equation*}
\caption{Values of $a(t)$\label{values of a}}
\end{table}

In particular, we note that $s_2(t)=\left\lceil (3t + 1)^{-1}\binom{t+3}{3}\right\rceil$
for all $t\in\NN$. Consequently, if the required base cases can be verified, i.e., if the  statements $\statement{A}_{K,s_1}(t)$ and $\statement{A}_{K,s_2}(t)$ are true for all $t \le t_0 = 82$, then $\sigma_s(\mathcal{C}_{3,n})$ and $\sigma_s(\mathcal{C}_{d,3})$ can be shown to be nondefective for all $n$, $d$, and $s$ by \cref{cor_nondefective}.
However, since in this case $t_0=27\cdot 3 + 1 = 82$, this requires challenging computations in vector spaces of dimensions up to $\binom{82 + 3}{3}=98\,770$, which we were unable to perform using standard software. Therefore, we describe our approach for proving them in some detail. In the remainder of this section, we let the vector space $V_t := V(t)$ for brevity.

In the literature \cite{AOP2009,BO,CO2012,abo,BCO2014,COV2014,AV}, a standard approach for proving the base cases of a Brambilla--Ottaviani lattice induction has emerged. It goes as follows. For proving that $\statement{A}_{K,s_b}(t)$ is true, where $\statement{A}_{i,s}(t)$ is defined as below \cref{defn_expdim} and $b=1,2$, a matrix $\mx{T}_{i,b}$ whose column span coincides with $A_{i,s_b}(t)$ is constructed. If the statement $\statement{A}_{K,s_b}(t)$ is true, then there is a Zariski-open subset of points in the $s_b(t)$-fold product of $X(t)$ such that 
\(
\dim A_{i,s_b}(t) = \operatorname{expdim} A_{i,s_b}(t).
\)
However, if the statement is false, then there is no such set of points. Consequently, it suffices to find one configuration of $s_b(t)$ points on $X(t)$ for which the corresponding space $A_{i,s_b}(t)$ has the expected dimension. These points should be chosen as beneficially as possible. For example, if we can choose points such that the matrix $\mx{T}_{i,b}$ has entries over the natural numbers, then we can bound its rank over $\kk$ from below by choosing a small prime number $P$ and computing the rank of $\mx{T}_{i,b}$ over the finite field $\ZZ_P$. This rank is computed efficiently by reducing $\mx{T}_{i,b}$ to row-echelon form using Gaussian elimination in $\ZZ_P$. If the rank of $\mx{T}_{i,b}$ over $\ZZ_P$ equals $\operatorname{expdim} A_{i,s_b}(T)$, then this implies that $\statement{A}_{i,s_b}(t)$ is true.

To construct aforementioned matrices $\mx{T}_{i,b}$, we choose coordinates on the ambient space $S^d V$ of the affine cone of $\mathcal{C}_{d,n}$ as follows. Take the coordinates $\{ x_0, \ldots, x_n \}$ on $V$. Then, a standard basis of $S^d V$ is the monomial basis 
\[
E_{n,d} = \{ x_{i_1} x_{i_2} \cdots x_{i_d} \mid 0 \le i_1 \le i_2 \le \cdots \le i_d \le n \}.
\]
Let $1 \le z_{i_1,\ldots,i_d} \le \binom{n+d}{d}$ denote the position of monomial $x_{i_1} \cdots x_{i_d} \in E_{n,d}$ in the lex-ordering of $E_{n,d}$, i.e., $x_{0}^d, x_{0}^{d-1} x_{1}, \ldots, x_{n}^d$. We have an isomorphism between $S^d V$ and $\kk^{\binom{n+d}{d}}$ via the bijection 
\[
 \nu_d : E_{n,d} \to \kk^{\binom{n+d}{d}}, \; x_{i_1} \cdots x_{i_d} \mapsto e_{z_{i_1,\ldots,i_d}}
\]
between the basis vectors $E_{n,d}$ of $S^d V$ and the standard basis $\{e_i\}$ on $\kk^{\binom{n+d}{d}}$ where $e_i$ has $1$ in position $i$ and zeros elsewhere. In this way, the product of $d$ linear forms 
\[
 \ell_i = \ell_{i,0} x_0 + \ell_{i,1} x_1 + \cdots + \ell_{i,n} x_n \in V
\]
would be represented practically in $\kk^{\binom{n+d}{d}}$ as 
\[
 \nu_d \left( \ell_1 \cdots \ell_d \right) = 
 \left(\prod_{k=1}^d \ell_{k,0}, \sum_{k=1}^d \ell_{k,1} \prod_{j\ne k} \ell_{j,0}, \ldots, \sum_{k=1}^d \ell_{k,n} \prod_{j\ne k} \ell_{j,n-1}, \prod_{k=1}^d \ell_{k,n}  \right).
\]

In the following subsections, it is helpful to keep in mind that the linear space $A_{i,s_b}(t)$ can be expressed as in \cref{eqn_simple_A}. This implies that it suffices to pick 
\begin{enumerate}[(i)]
 \item $\nabla^{i-1} s_b(t - \ell)$ points generically on $\widehat{X}(t) \cap \widehat{P}_k(t)$ for $k=1,\ldots,i$, and
 \item $\nabla^i s_b(t)$ points generically on $\widehat{X}(t)$
\end{enumerate}
to verify the truth of $\statement{A}_{K,s_b}(t)$.

In the next subsections, we fill in the details of this approach for induction on degree for $(n+1)$-ary forms and induction on dimension for cubics.

\subsection{An algorithm for induction on degree}
The base cases of the Brambilla--Ottaviani lattice induction for $(n+1)$-ary decomposable forms are proved as suggested in \cite[Remark 3.6]{abo} for $\alpha=0$. For $j = 1, \ldots, K(t)$, we consider the following linear subspaces $\widehat{P}_j(t)$ of $S^t V$, where $V$ is an $(n+1)$-dimensional vector space. As explained in \cref{sec_ind_degree_chow}, we first choose $\ell$ fixed linear forms; for computational efficiency, we choose $g_{j,1}, \ldots, g_{j,\ell} \in (V \cap \NN^{n+1})$, i.e., taking natural numbers as coordinates with respect to the standard basis $\{x_0, \ldots, x_n\}$. Then,
\[
 \widehat{P}_j(t) := g_{j,1} \cdots g_{j,\ell} S^{t - \ell} V
\]
for $t \ge \ell$. For brevity, we write $G_j = g_{j,1} \cdots g_{j,\ell}.$
Note that a point $q_j \in \widehat{P}_j(t) \cap \widehat{X}(t)$ can be expressed as $q_j := G_j F_j = g_{j,1} \cdots g_{j,\ell} f_{j,\ell+1} \cdots f_{j,t}$. Therefore, we see that adding $\widehat{P}_j(t)$ on both sides of 
\begin{dmath*}
 T_{q_j} \widehat{X}(t) = G_j f_{j,\ell+2}\cdots f_{j,t} V + \cdots + G_j f_{j,\ell+1} \cdots f_{j,t-1} V +
 g_{j,2}\cdots g_{j,\ell} F_j V + \cdots + g_{j,1} \cdots g_{j,\ell-1} F_j V
\end{dmath*} 
results in $T_{q_j} \widehat{X}(t) + \widehat{P}_j(t) = \widehat{P}_j(t) + g_{j,2}\cdots g_{j,\ell} F_j V + \cdots + g_{j,1} \cdots g_{j,\ell-1} F_j V$. This simplifies the approach from \cite[Remark 3.6]{abo} somewhat.

With the foregoing choice of subspaces $\widehat{P}_j(t)$ the statements $\statement{A}_{i,s_b}(t)$ can be verified with \cref{alg_quaternary}. It is a straightforward implementation of the standard approach outlined at the start of this section.

\begin{algorithm}
 \begin{enumerate}
  \item Choose a small prime number such as $P = 8191$.
  \item For $j=1,\ldots,K(t)$, choose $g_{j,1}, \ldots, g_{j,\ell} \in (V \cap \NN^{n+1})$ by randomly sampling the coordinates with respect to the standard basis $x_0, \ldots, x_n$ from the uniform distribution on $\{0,1,\ldots,P-1\}$. Compute $G_j = \prod_{\gamma=1}^\ell g_{j,\gamma}$. The space $\widehat{P}_j(t)$ is then generated by the columns of the matrix 
  \[
   \mx{P}_j = 
   \begin{bmatrix}
    \nu_t( x_0^{t-\ell} G_j ) & \nu_t( x_0^{t-\ell-1}x_1 G_j ) & \cdots & \nu_t( x_{n}^{t-\ell} G_j )
   \end{bmatrix}.
  \]
  Let $\mx{R}_1 = \begin{bmatrix} \mx{P}_1 & \cdots & \mx{P}_{K(t)} \end{bmatrix}$ be their horizontal concatenation.

  \item Select $\eta = \nabla^{K(t)} s_b(t)$ points $p_i = \prod_{\gamma=1}^t \ell_{i,\gamma}$ by randomly taking $\ell_{i,\gamma} \in V$. The coefficients of these $\ell_{i,\gamma}$'s are sampled from the uniform distribution on $\{0,1,\ldots,P-1\}$. As $\{x_0,\ldots,x_n\}$ is a basis of $V$, the tangent space $T_{p_i} \widehat{X}(t)$ is the span of all the columns of all the matrices
  \[
   \mx{G}_{i,\beta} = 
   \begin{bmatrix}
    \nu_t \left( x_0 \prod_{\gamma\ne\beta} \ell_{i,\gamma} \right) & \cdots & \nu_t \left( x_n \prod_{\gamma\ne\beta} \ell_{i,\gamma} \right)
   \end{bmatrix}, 
  \]
for $\beta = 1, \ldots, t$. Let $\mx{R}_2 = \begin{bmatrix}\mx{G}_{1,1} & \cdots & \mx{G}_{1,t} & \cdots & \mx{G}_{\eta,1} & \cdots & \mx{G}_{\eta,t} \end{bmatrix}$ be the horizontal concatenation of all these matrices.
  \item For each $j = 1, \ldots, K(t)$, select $\mu = \nabla^{K(t)-1} s_b(t-\ell)$ points $q_{i,j}$ in $\widehat{X}(t) \cap \widehat{P}_j(t)$ by randomly choosing $f_{i,j,1}, \ldots, f_{i,j,t-\ell} \in (V \cap \NN^{n+1})$ and setting 
  \[
   q_{i,j} = G_j \cdot F_{i,j}, \text{ where } F_{i,j} := \prod_{\gamma=1}^{t-\ell} f_{i,j,\gamma}.
  \]
 The coefficients of the linear forms $f_{i,j,\gamma}$ are sampled uniformly at random from $\{0,1,\ldots,P-1\}$. The tangent space to $\widehat{X}(t)$ modulo $\widehat{P}_j(t)$ at $q_{i,j}$ is given by the span of all the columns of all the matrices 
 \[
  \mx{S}_{i,j,\beta} = 
  \begin{bmatrix} 
   \nu_t \left( x_0 F_{i,j} \prod_{\gamma\ne\beta} g_{j,\gamma} \right) & \cdots & \nu_t \left( x_n F_{i,j}  \prod_{\gamma\ne\beta} g_{j,\gamma} \right)
  \end{bmatrix}
 \]
 for $\beta = 1, \ldots, \ell$. Let $\mx{R}_3$ be the horizontal concatenation of these matrices.
  \item Compute the rank of $\mx{T} = \begin{bmatrix} \mx{R}_1 & \mx{R}_2 & \mx{R}_3 \end{bmatrix}$ over the finite field $\ZZ_P$. If this rank equals $\operatorname{expdim} A_{K,s_b}(t)$, then the statement $\statement{A}_{K,s_b}(t)$ is true. Otherwise, no conclusion can be drawn: the field $\ZZ_P$ might be too small, the forms $g_{j,\gamma}$ or the points $p_i$ or $q_{i,j}$ might be unfortunately chosen, or the statement might be false.
 \end{enumerate}
\caption{Induction on degree for proving the base cases of the Brambilla--Ottaviani lattice presented in \cref{sec_ind_degree_chow} for $(n+1)$-ary decomposable forms.}
\label{alg_quaternary}
\end{algorithm}

\subsection{An algorithm for induction on dimension}
We present the algorithm only for cubics here to unburden the notation. It can be straightforwardly generalized to higher degrees. Nevertheless, it appears that the base cases for quartics might require an induction step length $\ell$ as large as $1536$, potentially necessitating computations in a space of dimension $\binom{4 \ell + 5}{4} = 59\,509\,031\,082\,501$, far beyond the reach of current computing infrastructure.

The base cases of the Brambilla--Ottaviani lattice induction for cubics can be handled similarly to section $5$ of \cite{AV} for $\alpha=0$. For $j=1,\ldots,K(t)$, let
\[
 U_j^\perp(t) := \langle x_{\ell\cdot(j-1)}, \ldots, x_{\ell j - 1} \rangle \subset V_t
\]
be the linear subspace spanned by the coordinates $x_{\ell(j-1) + i-1} \in \kk^{t+1}$ for $i=1,\ldots,\ell$, assuming that the coordinates on $V_t$ are $x_0,\ldots,x_{t}$. As subspaces $U_j(t) \subset V_t$ we select 
\[
U_j(t) = \kk^{t+1} / U_j^\perp(t).
\] 
Then, $\widehat{P}_j(t) = S^d U_j(t)$. Letting $q_j = k_j \ell_j m_j$ with $k_j, \ell_j, m_j \in U_j(t)$, we note that the tangent space to $\widehat{X}(t)$ at $q_{j}$ modulo $\widehat{P}_j(t)$ has a simple expression as in \cite[section 3.2]{AV}. Indeed, adding $\widehat{P}_j(t)$ to both sides of
\[
 T_{q_j} \widehat{X}(t) = (U_j(t) \oplus U_j^\perp(t)) \ell_j m_j + k_j (U_j(t) \oplus U_j^\perp(t)) m_j + k_j \ell_j (U_j(t) \oplus U_j^\perp(t))
\]
yields 
\(
 T_{q_j} \widehat{X}(t) + \widehat{P}_j(t) = \widehat{P}_j(t) + U_j^\perp(t) \ell_{j} m_{j} + k_{j} U_j^\perp(t) m_{j} + k_{j} \ell_{j} U_j^\perp(t).
\)

The algorithm for verifying $\statement{A}_{K,s_b}(t)$ was adapted from the optimized version of \cite[Algorithm 1]{AV}. It is briefly presented here as \cref{alg_cubics} for completeness. Further details can be found in section 5 of \cite{AV}.

\begin{algorithm}
\begin{enumerate}
 \item Choose a small prime number $P$ like $P = 8191$.
 \item Select $\eta := \nabla^{K(t)} s_b(t)$ points $p_i = k_i \ell_i m_i$ of $\widehat{X}(t)$ by randomly taking $k_{i}, \ell_i, m_i \in V_t$. The coefficients of these vectors are sampled uniformly from $0,\ldots,P-1$.
 The tangent space to $\widehat{X}(t)$ at $p_i$ is $V_t \ell_i m_i + k_i V_t m_i + k_i \ell_i V_t$. As $\{x_0,\ldots,x_t\}$ is a basis of $V_t$, $T_{p_i} \widehat{X}(t)$ equals the span of all the columns of the matrices
 \begin{align*}
  \mx{T}_{p_i,1} &= \begin{bmatrix} \nu_3( x_0 \ell_i m_i ) & \cdots & \nu_3( x_t \ell_i m_i) \end{bmatrix}, \\
  \mx{T}_{p_i,2} &= \begin{bmatrix} \nu_3( k_i x_0 m_i ) & \cdots & \nu_3(k_i x_t m_i) \end{bmatrix}, \text{ and } \\
  \mx{T}_{p_i,3} &= \begin{bmatrix} \nu_3( k_i \ell_i x_0 ) & \cdots & \nu_3( k_i \ell_i x_t ) \end{bmatrix}.
 \end{align*}
Let $\mx{R}_1$ denote the horizontal concatenation of this collection of matrices: $\mx{R}_1 := \begin{bmatrix} \mx{T}_{p_1,1} & \mx{T}_{p_1,2} & \mx{T}_{p_1,3} & \cdots & \mx{T}_{p_{\eta},1} & \mx{T}_{p_{\eta},2} & \mx{T}_{p_{\eta},3} \end{bmatrix}$.

\item For each $j=1,\ldots,K(t)$, select $\mu = \nabla^{K(t)-1} s_b(t-\ell)$ points in $\widehat{X}(t) \cap \widehat{P}_j(t)$ by sampling the nonzero coefficients of $k_{i,j}, \ell_{i,j}, m_{i,j} \in U_j(t)$ randomly from $0, \ldots, P-1$ and setting $q_{i,j} = k_{i,j} \ell_{i,j} m_{i,j}$. The tangent space to $\widehat{X}(t)$ at $q_{i,j}$ modulo $\widehat{P}_j(t)$ is generated by the set of all columns of the matrices
\begin{align*}
 \mx{T}_{i,j,1} &= \begin{bmatrix} \nu_3( x_{\ell(j-1)} \ell_{i,j} m_{i,j} ) & \cdots & \nu_3 ( x_{\ell j - 1} \ell_{i,j} m_{i,j} ) \end{bmatrix}, \\
 \mx{T}_{i,j,2} &= \begin{bmatrix} \nu_3( k_{i,j} x_{\ell(j-1)} m_{i,j} ) & \cdots & \nu_3 ( k_{i,j} x_{\ell j - 1} m_{i,j} ) \end{bmatrix}, \text{ and } \\
 \mx{T}_{i,j,3} &= \begin{bmatrix} \nu_3( k_{i,j} \ell_{i,j} x_{\ell(j-1)} ) & \cdots & \nu_3 ( k_{i,j} \ell_{i,j} x_{\ell j - 1} ) \end{bmatrix}.
\end{align*}
Let $\mx{R}_2$ denote the horizontal concatenation of this collection of matrices: $\mx{R}_2 = \begin{bmatrix} \mx{T}_{1,1,1} & \mx{T}_{1,1,2} & \mx{T}_{1,1,3} & \cdots & \mx{T}_{\mu,K(t),1} & \mx{T}_{\mu,K(t),2} & \mx{T}_{\mu,K(t),3} \end{bmatrix}$.
\item Let $Z_j = \{ z_{i_1,\ldots,i_3} \mid i_1, i_2, i_3 \in \{0,\ldots,n\}\setminus\{\ell(j-1),\ldots,\ell j - 1\} \}$, and set 
\[
Y = \{1,\ldots,N(t)\} \setminus \left( \cup_{j=1}^{K(t)} Z_j \right).
\]
\item Set $\mx{T} = \begin{bmatrix} \mx{R}_1 & \mx{R}_2 \end{bmatrix}$ and let $\mx{T}(Y)$ denote the matrix formed by taking only the rows of $\mx{T}$ at the indices in $Y$. Compute $\rank \mx{T}(Y)$ over $\ZZ_P$. If this rank equals
\[
 \operatorname{expdim} A_{K,s_b}(t) - \left| \cup_{j=1}^{K(t)} Z_j \right|,
\]
then $\statement{A}_{K,s_b}(t)$ is true. Otherwise no conclusion can be drawn: the prime $P$ might be too small, the points $p_i$ and $q_{i,j}$ might be unfortunately chosen, or the statement might be false.
\end{enumerate}
\caption{Induction on dimension for proving the base cases of the Brambilla--Ottaviani lattice presented in \cref{sec_ind_dimension_chow} for cubic decomposable forms.}
\label{alg_cubics}
\end{algorithm}

\subsection{Computing the polynomial represented by a rank-1 Chow decomposition}
Contrary to a computer algebra package like Macaulay2 \cite{M2}, C++ has no native support for multiplying multivariate polynomials over a finite field. For this reason, we implemented this operation ourselves for the special case of multiplying $d$ linear forms. A naive algorithm consists of noting that 
\[
 \ell_1 \cdots \ell_d = S^d( \ell_1 \otimes \cdots \otimes \ell_d )
\]
where $\ell_i$ are linear forms in $n+1$ variables and $S^d$ denotes the symmetrization operator in degree $d$. However, this implementation has a complexity of at least $(n+1)^d$ elementary operations due to the tensor product. For the largest case we need to handle ($n=3$ and $d=82$) that requires about $0.2 \cdot 10^{50}$ operations. The combined computing power of the $500$ most powerful supercomputers is approximately $10^{20}$ (floating point) operations per second, so even they would require the estimated age of the universe ($10^{10}$ years) a trillion times over. Hence we decided to search for a more efficient scheme. In the literature several advanced algorithms were proposed for multiplying two general multivariate polynomials; see, e.g., \cite{Moenck1976,CKY1989,Pan1994,BP1994}. Our setting is more specialized because we only need to multiply linear forms. In the end, we implemented a simple algorithm for computing $f \cdot \ell_i$, where $f \in S^k \ZZ_P^{n+1}$ is a polynomial of total degree $k$ in $n+1$ variables over $\ZZ_P$ and $\ell_i$ is the $i$th linear form. The idea is first to compute $f \otimes \ell_i$ instead, i.e., computing the rank-$1$ matrix represented by the coordinate representations of $f$ and $\ell_i$ with respect to the standard monomial bases of $S^k \ZZ_P^{n+1}$ and $S^1 \ZZ_P^{n+1}$ respectively. Then, the entries of this matrix corresponding to the same monomials are summed, resulting in the coordinate array of $f \cdot \ell_i$. The time complexity of our implementation is 
\[
\mathcal{O}\left(  (n+1) \cdot \binom{n+k}{k} \cdot \left( 2 + n \right)  \right).
\]
Using this algorithm, we compute $\ell_1 \cdots \ell_d$ sequentially as $( \cdots ( (\ell_1 \cdot \ell_2) \cdot \ell_3 ) \cdots ) \cdot \ell_d$. The implementation was quite efficient and outperformed Macaulay2's general polynomial multiplication routine \texttt{product} by a wide margin: multiplying $82$ random linear forms in $\ZZ_{8191}^4$ took $33.6$ seconds, averaged over $10$ runs, in Macaulay2 on a computer with $8$GB main memory and an Intel Core i7-5600U processor, while our C++ implementation took only $1.1$ seconds.

\subsection{Implementation details} 
\Cref{alg_cubics,alg_quaternary} were implemented in C$++$ and compiled with the GNU Compiler Collection with $-O3$ enabled.\footnote{Our implementation can be found along with the arXiv submission of this paper.} We used the Eigen v3 library \cite{Eigen} for the basic matrix type and the \texttt{Rank} function of FFLAS-FFPACK \cite{FFLAS} to compute the rank of matrices over finite fields. This last library requires a BLAS implementation; we used OpenBLAS \cite{OpenBLAS}. The code was executed on a server with $128$GB of main memory and two Intel Xeon E5-2697 v3 processors (a total of 28 physical cores with a clock speed of 2.6GHz). Some portions of the code were parallelized via OpenMP v3.1 directives.

\subsection{The results}

The main use of \cref{alg_cubics,alg_quaternary} is to complete the proof of \cref{main_result}.

\begin{proof}[Proof of \cref{main_result}]
Let $\ell = 27$, $K_0 = 3$, and $t_0 = 82$. Recall the $\ell$-quasiquadratic polynomials $s_1$ and $s_2$ from \cref{eqn_the_esses}. We deal with the two cases separately.

\vspace{0.5em}
\paragraph{\textit{Case 1: Quaternary forms.}} Take the Brambilla--Ottaviani lattice from \cref{sec_ind_degree_chow}. It follows from \cref{cor_nondefective} that it suffices to prove $\statement{A}_{K,s_b}(t)$ for all $2 \le t \le 82$ and $b=1,2$. Note that $\statement{A}_{K,s_b}(1)$ is always true. To this end, we applied \cref{alg_quaternary} to each of these cases. The algorithm, using $P=8191$, verified that each of these statements is true.

Our implementation of the algorithm produces a certificate that consists of the explicit coordinates of the linear forms $g_{j,\gamma}$, $\ell_{i,\gamma}$, and $f_{i,j,\gamma}$. For this choice of forms, reported in the certificate, the corresponding linear subspace $A_{K,s_b}(t)$ has the expected dimension. All certificates can be found at \url{https://nextcloud.cs.kuleuven.be/s/PdtR6ZPRfmmYHRa} and the authors' web pages. They constitute the computer-assisted proof of this case. An example certificate for the true subabundant statement $\statement{A}_{K,s_1}(5)$ looks like this:
\begin{verbatim}
Using random seed: 1452337571
Need a 56 x 60 matrix.
l_{0,0} = [7354 6394  862 7318]
l_{0,1} = [6008 7131 6458 3996]
l_{0,2} = [ 956 1407 7361  119]
l_{0,3} = [1659 1730 3153 6358]
l_{0,4} = [1861 3230 4474 6784]
l_{1,0} = [2581 5927 3361 5265]
l_{1,1} = [6076 3508  373 2488]
l_{1,2} = [4744 1652 3436  940]
l_{1,3} = [  65 1209 4285 6640]
l_{1,4} = [7483 5618 2000 4187]
l_{2,0} = [4138 6897 4991 5908]
l_{2,1} = [7470 2404 1374 7439]
l_{2,2} = [2454 6397 6616 4915]
l_{2,3} = [3309 7016 1544 7528]
l_{2,4} = [2433  571 1439  458]
Constructed T in 0.001s.
Computed the rank of the 56 x 60 matrix T over F_8191 in 0.001s.
Found 48 vs. 48 expected.
T_0(3, 5, 27) is TRUE (SUBABUNDANT)
\end{verbatim}
The final statement of each certificate is of the form $T_i(n, t, \ell)$, where $i = K(t)$, $n = 3$ and $\ell = 27$.
 
Producing these certificates was computationally very demanding. The most challenging case, $\statement{A}_{2,s_1}(81)$, requires first the construction of the $95,284 \times 112,844$ matrix $\mx{T}$ and then the computation of its rank over $\ZZ_{P}$. The first phase took nearly $5$ hours, while the rank computation completed in about $3$ hours and $21$ minutes. The total time to prove all the base cases was $4$ days, $21$ hours and $34$ minutes.

\vspace{0.5em}
\paragraph{\textit{Case 2: Cubics}} We take the Brambilla--Ottaviani lattice from \cref{sec_ind_dimension_chow}. By \cref{cor_nondefective} we should only prove the base cases for $1 \le t \le 82$ and $b=1,2$. Applying \cref{alg_cubics}, we produced certificates showing that these base cases are all true. These certificates contain the coordinates of the linear forms $k_i$, $\ell_i$, and $m_i$ for which the corresponding $A_{K,s_b}(t)$ has the expected dimension. The computations were again performed in $\ZZ_{8191}$.

Proving all base cases for cubics by \cref{alg_cubics} was significantly faster than the case of quaternary forms. The main reason is that the construction of the matrix $\mx{T}$ required much less computational effort. All matrices were constructed in less than $2$ minutes. The total time to construct the proofs of the base cases was only about $11$ hours and $36$ minutes.
\end{proof}


\end{document}